\documentclass[reqno]{amsart}     
\usepackage{pstricks,pst-node,pst-coil,pst-plot}   
\theoremstyle{plain}

\newtheorem{thm}{Theorem}[section]     
\newtheorem{theorem}[thm]{Theorem}     
\newtheorem{cor}[thm]{Corollary}

\newtheorem{lemma}[thm]{Lemma}     
\newtheorem{prop}[thm]{Proposition}

\theoremstyle{remark}      
 
\newtheorem{remark}[thm]{Remark}

\theoremstyle{definition}      
\newtheorem{defn}[thm]{Definition}

\def\al{{\alpha}}         
\def\be{{\beta}}         
\def\de{{\delta}}         
         
\def\om{{\omega}}         
\def\Om{{\Omega}}         
\def\la{{\lambda}}

\def\si{{\sigma}}

\def\ep{{\epsilon}}         
         
\def\th{{\theta}}         
\def\Th{{\Theta}}         
         
\def\phi{{\varphi}}

\let\pa\partial

\DeclareMathAlphabet{\doba}{U}{msb}{m}{n}

\gdef\mN{\doba{N}}

\gdef\mR{\doba{R}}         
\gdef\mS{\doba{S}}

\def\Vol{{\mathop{\rm Vol}}}     

\def\Scal{{\mathop{\rm Scal}}}     
\def\dist{{\mathop{\rm dist}}}

\def\d2dt{\frac{d^2}{dt^2}}

\def\eref#1{{\rm (\ref{#1})}}   

\def\Rmax{R_{\textrm{max}}}

\let\<\langle 
\let\>\rangle

\long\def\ignorethis#1{}

\newdimen\templaenge

\def\Atbox#1#2{\setbox0\hbox{$\displaystyle #1$}\templaenge=\textwidth\advance\templaenge by -\wd0%
\setbox1\hbox{$#2$}\advance\templaenge by -\wd1%
$$#1\hbox{\kern\templaenge$#2$\hss}$$\par\bigbreak}

\newtheorem*{subcaseI.1}{Subcase I.1}
\newtheorem*{subcaseI.2}{Subcase I.2}

\newtheorem*{subcaseII.1}{Subcase II.1}
\newtheorem*{subsubcaseII.1.1}{Subsubcase II.1.1}
\newtheorem*{subsubcaseII.1.2}{Subsubcase II.1.2}
\newtheorem*{subcaseII.2}{Subcase II.2}

\def\detwo{\de_0}

\def\lds{\Gamma_n^{\rm lds}}

\def\conc{{\rm Conc}}

\def\riem{\mathcal{R}}
\newcommand{\definedas}{\mathrel{\raise.095ex\hbox{\rm :}\mkern-5.2mu=}}

\begin{document}     
%%%%%%%%%%%%%%%%%%%%%%%%%%%%%%%%%%%%%%%%%%%%%%%%%%%%%%%%%%%%%%%%%
%%%%%%%%%%%%%%%%%%%%%%%%%%%%%%%%%%%%%%%%%%%%%%%%%%%%%%%%%%%%%%%%%

\title{Relative Yamabe invariant and c-concordant metrics}     
 
%\author{Bernd Ammann} 
%\address{NWF I -- Mathematik \\ 
%Universit\"at Regensburg \\
%%93040 Regensburg \\  
%Germany}
%\email{bernd.ammann@mathematik.uni-regensburg.de}

%\author{Mattias Dahl} 
%\address{Institutionen f\"or Matematik \\
%Kungliga Tekniska H\"ogskolan \\
%100 44 Stockholm \\
%Sweden}
%\email{dahl@math.kth.se}
 
\author{Emmanuel Humbert} 
\address{Institut \'Elie Cartan, BP 239 \\ 
Universit\'e de Nancy 1 \\
54506 Vandoeuvre-l\`es-Nancy Cedex \\ 
France}
\email{humbert@iecn.u-nancy.fr}

\begin{abstract}
We prove  a surgery formula for the relative Yamabe invariant with  several
applications.  In particular, we study a Yamabe invariant defined on the
set of concordance classes of metrics.    
\end{abstract}     

% 
% 35J60   Nonlinear PDE of elliptic type
% 35P30   Nonlinear eigenvalue problems, nonlinear spectral theory for PDO
% 57R65   Surgery and handlebodies
% 58J50   Spectral problems; spectral geometry; scattering theory
% 58C40   Spectral theory; eigenvalue problems
%
% Old codim2 {53C27 (Primary) 55N22, 57R65 (Secondary)}
% Old mu2-MSC%%%% 53A30, 35J60(Primary) 35P30, 58J50, 58C40 (Secondary)   

\date{July 17th, 2008}


\maketitle     

\tableofcontents     

%%%%%%%%%%%%%%%%%%%%%%%%%%%%%%%%%%%%%%%%%%%%%%%%%%%%%%%%%%%%%%%%%%%%%%%%%%%%%%%%%%%%%%%%%%%%%%%%%%%%%%%%%%

%%%%%%%%%%%%%%%%%%%%%%%%%%%%%%%%%%%%%%%%%%%%%%%%%%%%%%%%%%%%%%%%%%%%%%%%%%%%%%%
\section{Introduction}

\noindent
% In the whole paper, the manifolds  are assumed to be
%oriented. If $M$ is a manifold, $-M$ denotes the same manifold with the
%opposite orientation. 
Let $M$ be a compact $n$-manifold ($n \geq 3$) without boundary and a conformal class of metrics $C$ on $M$.  The {\it Yamabe constant} is defined as
$$\mu(M,C) = \inf  \int_M \Scal_g dv_g.$$  
where the infimum is taken over the metrics $g \in C$ such that
$\Vol_g(M)=1$. For any metric $g$ on $M$ and any $u \in C^{\infty}(M)$,  let 
$$J^n_{M,g} (u) = \frac{ \int_{M} u L_g(u) dv_g}{\left(\int_M
  |u|^{\frac{2n}{n-2} } dv_g \right)^{\frac{n-2}{n}}}$$
\noindent Here, $L_g$ is the  {\it conformal Laplacian} or {\it Yamabe
  operator} and is defined by  
$$L_g = \frac{4(n-1)}{n-2} \Delta_g + \Scal_g.$$
Then, it is well known that (see \cite{lee.parker:87,aubin:98,hebey:97}) 

\begin{eqnarray} \label{variational}
\mu(M,[g]) = \inf_{u\in C^{\infty}(M); u \not=0} J^n_{M,g}(u).
\end{eqnarray}
The Yamabe constant has been introduced by Yamabe in 1960 while attempting
to find metrics of constant scalar curvature in a conformal class. He was
able to show that the infimum in the definition of $\mu$ is always
attained. Unfortunately, there was a gap in his proof which was repaired by
Tr\"udinger (1968), Aubin (1976) and Schoen (1984). For a survey on this
problem, the so-called {\it Yamabe problem}, the reader may refer to
\cite{lee.parker:87,aubin:98,hebey:97}. \\

\noindent Now, let 
$$\sigma(M)= \sup \mu(M,C)$$
where the supremum runs over the sets of all conformal classes $C$ of metrics. 
Aubin proved in \cite{aubin:76} that for all $C$,
 $\mu(M,C) \leq \mu(\mS^n) = n(n-1) \om_n^{\frac{2}{n}}$ where
$\mS^n$ denotes the $n$-sphere $S^n$ equipped with its standard metric. Its
volume is denoted by $\om_n$. 
This
implies that $\sigma (M) \leq \sigma(S^n) = \mu(\mS^n)$ and  hence,
$\sigma(M)$  is well-defined and depends only on the differentiable
manifold $M$. It is called the {\it Yamabe invariant}.  \\

\noindent The critical points of the functional $g \to \int_M \Scal_g dv_g$
among the metrics $g$ with $\Vol_g(M)=1$ are Einstein metrics. Besides, one
can check that $\si(M) >0$ if and only if there exists a metric $g$ on $M$
with positive scalar curvature. Hence, the study of $\sigma(M)$ is connected to the difficult and still unsolved problems
to determine all manifolds admitting einstein metrics and admitting
metrics of positive scalar curvature. 
This explains why the Yamabe invariant has attracted so much interests in
the last decades. For more informations and references on $\si$, the reader
may consult \cite{ammann.dahl.humbert:08}. The Yamabe invariant turns out to be
very difficult to compute explicitly and the value of $\sigma$ is known
only for
very few manifolds (see again \cite{ammann.dahl.humbert:08}). 
A natural way to go further in its study is to use surgery techniques,
whose power is demonstrated in the papers of Gromov-Lawson
\cite{gromov.lawson:80} and Schoen-Yau \cite{schoen.yau:79}.\\

\noindent  We focus here on the 
following surgery 
theorem, proved in \cite{ammann.dahl.humbert:08}, which plays a
central role in the whole paper:

\begin{theorem} \label{adh} {\bf(Ammann-Dahl-Humbert; 2008)}
Let $(M,g)$ be a compact $n$-dimensional ($n \geq 3$) Riemannian
manifold and let $M^\#$ be obtained from $M$ by a surgery of dimension $k
\in \{ 0,\cdots,n-3\}$. Then, there exists constants $\beta_{n,k}>0$ with
$\beta_{n,0} = +\infty$ depending only on $n$ and $k$ and metrics $(g_\th)_{\th >0}$ on $M^\#$ such that  
$$\lim_{\th \to 0} \mu(M^\#,[g_\th]) \geq \min(\mu(M,[g]), \beta_{n,k}).$$
In particular, 
$$\sigma(M^\#) \geq \min(\sigma(M), \beta_n)$$
where $\beta_n = \min_{k \in \{0,\cdots,n-3\}} \beta_{n,k}$.
\end{theorem}

\noindent Now, let $\Om$ be a $(n+1)$-manifold with boundary $M$. 
If $g$ is a metric on $\Om$, we denote by $\partial g$ the metric induced by $g$
on $M$. If $C$ is a conformal class, then $\partial C \definedas \{ \partial
g | g \in C \}$. In particular, if $g$ is a metric on $\Om$, $\partial [g]
= [ \partial g ]$.  Let us fix a conformal class $C$ (resp. $\bar{C}$) on $M$ (resp. $\Om$) such that $\partial \bar{C} =C$. Again, we can define 
$$\mu(\Om,\bar{C}; M,C) = \inf  \int_\Om \Scal_g dv_g$$  
where the infimum is taken over the metrics $g \in \bar{C}$ for which the boundary $M$ is minimal and for which $\Om$ has volume $1$.
This number is called the {\it relative Yamabe constant} and by
Escobar \cite{escobar:92} 
$$\mu(\Om,\bar{C}; M,C) \leq \mu(\mS^{n+1}_+,\mS^n )= 2^\frac{2}{n+1}
\mu(\mS^{n+1}) =  2^\frac{2}{n+1}n(n+1) \om_{n+1}^\frac{2}{n+1}.$$
As in the case of manifolds without boundary, for all metric $g$ on $\Om$ such that $\partial g= h$ and all $u \in C^{\infty}(\Om)$, we set 
$$J^{n+1}_{\Om,g} (u) = \frac{ \int_{\Om} u L_g(u) dv_g +  \int_M H_gu^2 dv_h}{\left(\int_\Om
  |u|^{\frac{2(n+1)}{n-1} } dv_g \right)^{\frac{n-1}{n+1}}}$$
where $H_g$ is the mean curvature of the boundary $M$ with respect to the
metric $g$. Then, it is well known that 

\begin{eqnarray*} 
\mu(\Om,[g];M,[h]) = \inf_{u\in C^{\infty}(\Om); u \not=0} J^{n+1}_{\Om,g}(u).
\end{eqnarray*} 

\noindent If in addition $M$ is minimal for the metric $g$, 
then 
\begin{eqnarray} \label{variational_rel}
\mu(\Om,[g];M,[h]) = \inf_{u\in C^{\infty}(\Om); u \not=0; \pa_{\nu}u=0} J^{n+1}_{\Om,g}(u),
\end{eqnarray} 
\noindent where $\nu$ is the outer normal unit vector field on $M$. \\

\noindent Escobar \cite{escobar:92} studied a Yamabe type problem concerning this conformal invariant. More precisely, he studied the problem of finding in a conformal class metrics with constant scalar curvature for which the boundary is minimal. He proved  

\begin{theorem} \label{escobar} {\bf (Escobar; 1992)}
Let $(\Om,g)$ be a compact $(n+1)$-dimensional manifold for which the boundary $M$ is minimal. 
Assume that $$\mu(\Om,[g]; M,[\partial g]) < \mu(\mS^{n+1}_+,\mS^n ).$$ 
Then, $\mu(\Om,[g]; M,[\partial g])$ is attained. In other words, there exists a metric $g'$ conformal to $g$ for which 
$\Vol_{g'}(\Om)=1$, $\partial g' \in [\partial g]$, $M$ is minimal in $(\Om,g')$ and such that 
$$\mu(\Om,\bar{C}; M,C) = \int_\Om \Scal_{g'} dv_{g'}.$$
\end{theorem}
\noindent One can verifies that the metric $g'$ of the above theorem has constant scalar curvature. Since it is conformal to $g$,
it has the form $g'=u^{\frac{4}{n-1}} g$ for some positive smooth function
$u$ on $\Om$ that we can normalized by $\int_{\Om} u^p dv_g =1$ where $p \definedas \frac{2(n+1)}{n-1}$.  Then, 
$$\mu (\Om,[g]; M,[\partial g]) = J^{n+1}_{\Om,g} (u).$$
Writing the Euler equation of $u$, we obtain:
\[\left\{ \begin{array}{ccc} 
L_{g'} u =  \mu(\Om,\bar{C}; M,C) u^{p-1}  & \hbox{ on } & \stackrel{\circ}{\Om} \\
\partial_\nu u = 0 & \hbox{ on } & \partial \Om =M.
\end{array} \right. \]
 
\noindent Now, as in the case of manifolds without boundary, one defined the {\it relative Yamabe invariant} by 
$$\sigma(\Om;M,C) = \sup_{C'} \mu(\Om,C' ; M,C)$$
where the supremum runs over all conformal classes of metrics $C'$ on $\Om$ for which $\partial C'= C$.
This invariant is related to the topology of the set of metrics with positive scalar curvature. It was studied for example by Akutagawa and Botvinnik. In particular, they proved in \cite{akutagawa.botvinnik:02a} the following result:

\begin{theorem} \label{ab} {\bf (Akutagawa, Botvinnik; 2002)} 
Let $\Om_1$, $\Om_2$ be $(n+1)$-dimensional manifolds with respective boundaries $M_1 \amalg M$ and $M_2 \amalg M$ ($M_1$ and $M_2$ are possibly empty). Let $C_1,C_2,C$ be conformal classes of metrics respectively on $M_1$, $M_2$ and $M$ and let $\Om$ be the manifold with boundary $M_1 \amalg M_2$ obtained by gluing $\Om_1$ and $\Om_2$ along $M$. Assume that $\sigma(\Om_i;M_i \amalg M, C_i \amalg C) >0$ for $i=1,2$. Then, $\sigma(\Om; M_1 \amalg M_2, C_1 \amalg C_2) >0$ where by convention, $\sigma (\Om; \emptyset) = \sigma(M)$.   
\end{theorem}
\noindent This paper aims to obtain a surgery formula for the relative Yamabe invariant
similar to the one in Theorem \ref{adh} for the standard Yamabe invariant. More precisely, we prove

\begin{theorem} \label{main}
Let $n\in \mN$, $n \geq 2$ and $(\Om,g)$ be a compact $(n+1)$-dimensional  
Riemannian manifold with boundary $M$. We set $h \definedas
\partial g$.  Let also $k \in
\{0,\cdots, n-2 \}$ and $M^\#$ be obtained from $M$ by a surgery of
dimension $k$.  We denote by  $\Om^\#$  the
manifold with boundary $M^\#$ obtained
from $\Om$ by attaching the corresponding $(k+1)$-dimensional handle along
$M$.
Then, there exist some constants $\alpha_{n,k} >0$ depending only on $n$
and $k$ and a sequence of metrics $(g_\th)_{\th>0}$  on $\Om^\#$ such that,
setting $h_\th = \partial g_\th$
\begin{eqnarray} 
\lim_{\th \to 0} \mu(\Om^\#, [g_\th]; M^\#,[h_\th]) \geq \min(\mu(\Om,[g];M,[h]),\al_{n,k}).
\end{eqnarray}
If in addition, $ n \geq 3$ and $k \leq n-3$, the metrics $h_\th$
coincide with the metrics given by Theorem \ref{adh}. In
other words, there exists a constant $\beta_{n,k}>0$ depending only on $n$
and $k$ (the same as in
Theorem \ref{adh}) such that  
\begin{eqnarray} 
\lim_{\th \to 0} \mu(M^\#,[h_\th]) \geq \min(\mu(M,[h]),\beta_{n,k}).
\end{eqnarray} 
Moreover, for $k=0$, we can assume that
\begin{eqnarray} \label{k=0}
\alpha_{n,0} = \beta_{n,0} = + \infty.
\end{eqnarray}
\end{theorem}

\noindent This theorem is an equivalent of Theorem \ref{adh} for manifolds
with boundary. Adapting such surgery results on manifolds with boundary has
already be done and
Theorem \ref{main} is in the spirit of the results in \cite{gajer:87},
\cite{dahl:06} or \cite{andersonn:08}.  
A first corollary of our theorem
is:

\begin{cor} \label{main_cor}
Let $n \geq 2$ and let $\Om$ be a $(n+1)$-dimensional compact manifold with boundary $M$ and
let $\Om^\sharp$ be obtained by adding a $(k+1)$-dimensional handle on $M$
for some $k \in \{0,\cdots,n-2\}$. Let $C$ be a conformal class on $M$. We note $M^\#= \partial \Om^\#$ which is obtained from $M$ by a surgery of dimension $k$. Then there exists a conformal class $C^\#$ on $M^\#$ such that 
$$\si(\Om^\#;M^\#,C^\#) \geq \min( \si(\Om;M,C), \al_n)$$
where 
$$\al_n \definedas \min_{k \in \{0, \cdots n-2\}} \al_{n,k}$$ 
and where $\al_{n,k}$ is as in Theorem \ref{main}.
If in addition, $n \geq 3$ and $k \not= n-2$, for all $\ep >0$, we can choose $C^\#$ so that  
$$\mu(M^\#, C^\#) \geq \min(\mu(M,C),\beta_n   ) - \ep$$
where 
$$\beta_n \definedas \min_{k \in \{0, \cdots n-3\}} \beta_{n,k}$$ 
and where $\beta_{n,k}$ is as in Theorem \ref{adh}.
\end{cor}
 
\noindent Among immediate consequence of Corollary \ref{main_cor}, we can
observe that since $\alpha_{n}, \beta_n>0$, 
we obtain a new proof of main Theorem in \cite{gajer:87}. Note that there
was a gap in the 
proof of Gajer which was repaired by Walsh \cite{walsh:08}. 
Another
consequence of Corollary \ref{main_cor} is the main result concerning
relative Yamabe invariant in Schwartz
\cite{schwartz:08} 
which in particular implies that handlebodies have maximal relative Yamabe
invariant among manifolds with boundary.

%invariant in \cite{schwartz:08}. In particular, the applications of these two results can be also deduced from Corollary \ref{main_cor}.\\
 
\noindent  We now explain a subtler consequence of our results. Let $g,g'$ be metrics with positive scalar curvature on
$M$. We say that that $g,g'$ are {\it conformally concordant} or {\it $0$-concordant}
if there exists a metric $G$ on $M \times [0,1]$ conformal to a metric
with positive scalar
curvature for which $\partial M \times [0,1]$ is minimal and such that
$\partial [G] = [g] \amalg [g']$. It follows from
\cite{akutagawa.botvinnik:02a}  that
``to be concordant'' is an equivalence relation. The set of equivalence
classes is denoted by $\conc_0(M)$.

\noindent We now define 
\[  \sigma'' \definedas \left| \begin{array}{ccc}
 \conc_0(M) & \to & ] \infty,\sigma(M)] \\
 C & \mapsto &  \min\left(\sup_{g \in C} (\mu(M,[g])),\beta_n\right)
\end{array} \right. \]
where $\beta_n$ is as in Corollary \ref{main_cor} so that 
$$\min(\si(M), \beta_n) = \sup_{C \in  \conc_0(M)} \si''(C).$$  
A hard open question is to know whether $\si$ is attained or not.  A first
step in this direction could be to study whether the supremum above is
attained or not. This is the main motivation here to introduce $\si''$. 
As an application of Theorem \ref{main}, we prove in Section \ref{concordant}

\begin{theorem} \label{application} 
Let $M,N$ be compact $n$-manifolds such that $N$ is obtained from $M$ by a
finite sequence of surgeries of dimension $k \in \{2,\cdots,n-3\}$. 
Then 
$$\sigma'' (\conc_0(M) )= \sigma''(\conc_0(N)).$$
\end{theorem}

\smallskip

\noindent {\bf Acknowledgements:} The author is very grateful to Bernd
Ammann, Mattias Dahl and Julien Maubon for many helpful discussions and
comments. 

%%%%%%%%%%%%%%%%%%%%%%%%%%%%%%%%%%%%%%%%%%%%%%%%%%%%%%%%%%%%%%%%%%%%%%%%%%%%%%%%%%%%%%%%%%%%

%%%%%%%%%%%%%%%%%%%%%%%%%%%%%%%%%%%%%%%%%%%%%%%%%%%%%%%%%%%%%%%%%%%%%%%%%%%%%%%%%%%%%%%

%%%%%%%%%%%%%%%%%%%%%%%%%%%%%%%%%%%%%%%%%%%%%%%%%%%%%%%%%%%%%%%%%%%%%%
\section{Surgery} \label{surgery}
%%%%%%%%%%%%%%%%%%%%%%%%%%%%%%%%%%%%%%%%%%%%%%%%%%%%%%%%%%%%%%%%%%%%%%
In this section, we prove Theorem \ref{main}. In this goal, we give some
basic facts on the double of manifolds with boundary which will be used later. We also give the definitions of surgery and attachment of
handles. The last Paragraph \ref{surgery_cyl} is devoted to the proof of
two lemmas which will be helpful in Section \ref{concordant}. 

\subsection{The double of a manifold with boundary} \label{double}
Let $\Om$ be a compact $(n+1)$-dimensional manifold with boundary $M$. The {\it
  double of $M$} is the compact  manifold without boundary $X \definedas \Om \cup_M
\Om$ obtained by gluing two copies of $\Om$ along their common
boundary. 
%For instance, if $M$ is a compact manifold without boundary of dimension
%$n$, the double of $[0,1] \times M$ is $S^1 \times M$. 
Let $g$ be a
metric on $\Om$ and let $h\definedas \partial g$ be the induced metric on
the boundary $M$. Assume that $g$ is a product metric near the boundary
$M$, i.e. that $g$ has the form $g= h + ds^2$, $s$ being the distance to $M$. Then $g$ extends naturally 
to a smooth metric $\bar{g}\definedas g \cup g$ on $X$. We will need the following basic results:

\begin{prop} \label{basic}
%\begin{enumerate} 
%\item Assume that $\Om$ is simply connected and that $M$ is connected. Then, $X$ is simply connected. 
%\item 
Let $u \in C^\infty(\Om)$, $u \geq 0$  be a non-negative function which satisfies:
\[\left\{ \begin{array}{ccc} 
L_g u =  \la u^p  & \hbox{ on } & \stackrel{\circ}{\Om} \\
\partial_\nu u = 0 & \hbox{ on } & \partial \Om =M,
\end{array} \right. \]
for some $\la \in \mR$ and some $p \geq 1$. The function $\bar{u}=u \cup u$ is smooth on $X$. 
%\item If  $Y_{[g]}(\Om,M,[h]) \leq 0$ then 
%$$\mu(X,[\bar{g}]) =2^{\frac{2}{n}}Y_{[g]}(\Om,M,[h]).$$
%\item We have 
%$$Y_{[g]}(\Om,M,[h]) >0 (\hbox{resp. } =0, <0) \Longleftrightarrow \mu(X,[\bar{g}]) >0(\hbox{resp. } =0, <0).$$
%\end{enumerate}
\end{prop}

\begin{proof}
 Just notice that  $\bar{u} \in C^1(M)$ and satisfies 
$L_{\bar{g}} \bar{u} = \la  \bar{u}^p$ weakly on $X$. Then, $\bar{u} \in C^{\infty}(X)$ by standard elliptic regularity theorems.
\end{proof}

\begin{prop} \label{basic2} 
We have 
$$2^{\frac{2}{n+1}} \mu(\Om,[g]; M,[h]) = \inf 
J^{n+1}_{X,\bar{g}} (u)$$
where the infimum runs over
the non-zero functions $u \in C^{\infty}(X)$ such that $\partial_\nu u
\equiv 0$ on $M$, $\nu$ being any normal vector field on $M$. 
\end{prop}
\noindent The proof  easily follows from \eref{variational}, \eref{variational_rel} and the fact that the mean curvature $H_g$ vanishes on 
$M$.

%%%%%%%%%%%%%%%%%%%%%%%%%%%%%%%%%%%%%%%%%%%%%%%%%%%%%%%%%%%%%%%%%%%%%%%%%%%%%%%%%%%%%%%%%%%%%%%%%%%%
\subsection{Surgeries and attachments of handles} \label{surgatt}
\noindent Let $M$ be a $n$-dimensional manifold and let $k$ be an integer such that 
$0 \leq k \leq n-1$. We assume that an embedding $f:S^k \times B^{n-k} \to \stackrel{\circ}{M}$ is given.
Then, $M \setminus f(S^k \times B^{n-k})$ is a manifold whose boundary is 
diffeomophic to $S^{k} \times S^{n-k-1}$. We then construct 
$$M^\#_f \definedas (M \setminus f(S^k \times B^{n-k}) \cup_{f(S^k \times S^{n-k-1})} (\overline{B^{k+1}} \times S^{n-k-1}).$$ 
We say that {\it $M^\#_f$ is  obtained from $M$ by a surgery of dimension $k$ along $f(S^k \times B^{n-k})$}.
If $M$ has a boundary, we say that {\it $M^\#_f$ is  obtained from $M$ by an interior surgery of dimension $k$} to emphasise the fact that nothing happens on the boundary. \\

\begin{remark} \label{surgery_cancel}
Observe that 
$$M= M^\#_f \setminus  (B^{k+1} \times S^{n-k-1}) \cup_{f(S^k \times S^{n-k-1})} \overline{f(S^k \times B^{n-k})}.$$
In particular, $M$ is obtained from $M^\#_f$ by a $(n-1-k)$-surgery that we will call the {\it dual surgery} of the surgery given by $f$. 
\end{remark}

\noindent Now, let $\Om$ be a $(n+1)$-dimensional differentiable manifold
whose boundary is $M$ and attach the disk $D^{k+1} \definedas \overline{B^{k+1}} \times
\overline{B^{n-k}}$ along $f(S^k \times B^{n-k}) \subset M$ using
$f$. Smoothing the corners, we get a new manifold 
$$\Om^\#_f \definedas \Om \cup_f (\overline{B^{k+1}} \times \overline{B^{n-k}})$$
with $\partial \Om^\#_f = M^\#_f$. We say that {\it $\Om^\#_f$ is obtained from $\Om$ by attachment of a handle of dimension $k+1$}.
The handle corresponding to the dual surgery of the one given by $f$ will be called the {\it the dual handle} of $D^{k+1}$ \\

\noindent Assume that near the boundary $M$ of $\Om$ we have some trivialisation $\Om \sim (M \times ]-2,0])$. 	In other words, $M= \partial \Om$ is identified to $M \times \{ 0 \}$. 
We define the half-balls and the half-spheres
\begin{eqnarray*}
B^{m+1}_- \;( \hbox{ resp. }  B^{m+1}_+ )& \definedas & \{ (y_1,\cdots,y_{m+1}) \in B^{m+1} \subset \mR^{m+1} \; | \; y_{m+1} \leq 0(\hbox{ resp. } \geq  0 )\} \\
S^m_-\; ( \hbox{ resp. }  S^m_+) &\definedas &  \{ (y_1,\cdots,y_{m+1}) \in
S^m \subset \mR^{m+1} \; | \; y_{m+1} \leq 0 (\hbox{ resp. } \geq  0) \}. \\
\end{eqnarray*}
and we set 
\[ F : \left| \begin{array}{ccc}
S^k \times \overline{B^{n+1-k}_-} & \to & M \times ]-2,0] \\
(x,(y_1,\cdots, y_{n+1-k}) ) & \mapsto & (f(x,y'),y_{n+1-k})\\
\end{array} \right. \]
where $y'=(y_1,\cdots, y_{n-k}) \in B^{n-k}$.
Clearly, $F$ is a smooth embedding in $\Om$  such that $F_{/S^k \times
  B^{n-k} } = f$ where $B^{n-k}$ is seen as a subset of $B^{n+1-k}_-$
writing that $B^{n-k} = \{ (y_1,\cdots y_{n+1-k} ) \in B^{n+1-k}_- |
y_{n+1-k} = 0\} $. 

\noindent Set now    
  $$\tilde{\Om}_F \definedas (\Om \setminus F(S^k \times \overline{B^{n+1-k}_-} )) \cup (\overline{B^{k+1}} \times S^{n-k}_-)_{/\sim}$$
where $\sim$ means that we glue the boundaries. It is straightforward to
see that $\tilde{\Om}_F$ is diffeomorphic to $\Om^\#_f$.
In this way, attaching a handle is view as a ``half-surgery'' on
$\Om$. This was also the viewpoint adopted by Ole
Andersonn \cite{andersonn:08} in his thesis.

%\mnote{est-ce que ça dépend des coordonnées?}

%%%%%%%%%%%%%%%%%%%%%%%%%%%%%%%%%%%%%%%%%%%%%%%%%%%%%%%%%%%%%%%%%
\subsection{Connected sum along a submanifold of manifolds with boundary} \label{connected_sum}
Assume first 
that $(M_1,h_1)$ and $(M_2,h_2)$ are Riemannian  manifolds without boundary of dimension $n$ and that
$W$ is a compact manifold of dimension~$k$. Let embeddings $W
\hookrightarrow M_1$ and $W \hookrightarrow M_2$ be given. We assume
further that the normal bundles of these embeddings are trivial. 
Removing tubular neighborhoods of the images of $W$ in $M_1$ and
$M_2$, and gluing together these manifolds along their common
boundary, we get a new compact manifold $M^\# \definedas M_1 \cup_W M_2$, called {\it the connected sum of
$M_1$ and $M_2$ along $W$}. Notice that $M^\#$ depends on the trivialisation of the normal bundles. 
Surgery as explained in last paragraph \ref{surgatt} is a special case of this construction: if $M_2=S^n$, 
$W=S^k$ and if $S^k\hookrightarrow S^n$ is the standard embedding,
then $M^\#$ is obtained from $M_1$ from a $k$-dimensional surgery along 
$S^k\hookrightarrow M_1$. For more informations on this construction, see \cite{ammann.dahl.humbert:08}.\\

\noindent In this paper, we need to adapt this construction to the case of manifolds with boundary. Let
$(\Om_1,g_1)$,  $(\Om_2,g_2)$ be $(n+1)$-dimensional Riemannian manifolds
with respective boundaries  $M_1$ and $M_2$. We denote by $h_i$ ($i=1,2$)
the trace of $g_i$ on $M_i$ i.e. $\partial g_i= h_i$.  Let $W$ be a compact
manifold of dimension $k$. If $W$ embedds in $\stackrel{\circ}{\Om}_1$ and
$\stackrel{\circ}{\Om}_2$, 
then we can proceed exactly as in the case of manifolds without boundary explained above and we obtain a new manifold $\Om^\# \definedas \Om_1 \cup_W \Om_2$ called again the connected sum of $\Om_1$ and $\Om_2$ along $W$. Obviously, $\partial \Om^\#= M_1 \amalg M_2$. 
In the case where $\Om_2=S^n$, 
$W=S^k$ and if $S^k\hookrightarrow S^n$ is the standard embedding,
then $\Om^\#$ is obtained from $\Om_1$ by  an interior $k$-dimensional surgery.

\noindent Now, assume that $W$ embedds into the boundaries $M_i$ of $\Om_i$. Let us make it precise now.   
We assume that some smooth embeddings $\bar{w}_i: W \times \mR^{n+1-k} \to T\Om_i$ $i=1,2$ are given. In what follows, we identify $\mR^{n-k}$ with $\mR^{n-k} \times \{0 \} \subset \mR^{n+1-k}$. 
%We also denote by $V$ the orthogonal complement (for the usual scalar product) of $\mR^{n-k}$ in $\mR^{n+1-k}$ i.e. $V \definedas %\{0 \} \times \mR$. 
We make the following additional assumptions of  $\bar{w}_i$ :

\begin{itemize}
\item First, we assume that $\bar{w}_i$ restricted to $W \times \mR^{n-k}$ embedds in $TM_i \subset T \Om_i$. 
\item Then, we want that $\bar{w}_i$ restricted to $W \times \{ 0 \}$ maps to the zero section of $TM_i$ (which we
identify with $M_i$) and thus gives an embedding $W \hookrightarrow M_i \subset \Om_i$. The
image of this embedding is denoted by $W_i'$.
\item Further we assume that
$\bar{w}_i$ restrict to linear isomorphisms $\al_p: \{ p \} \times \mR^{n+1-k}
\to N_{\bar{w}_i(p,0)} W_i'$ for all $p \in W$. Here $N W_i'$
denotes the normal bundle of $W_i'$ defined using $g_i$. In addition, we assume that 
$\al_p$ restricted to $\{p\} \times \mR^{n-k}$ is an isomorphism onto
$N_{\bar{w}_i(p,0)} W_i' \cap TM_i$. We can assume also that $\bar{w}_i(
\{p \} \times (0,\cdots,0,1))$ denotes the outer normal unit vector at $p$.
\end{itemize} 

\noindent  Now, we set $w_i \definedas \exp^{g_i} \circ \bar{w}_i$. This gives
embeddings $w_i: W \times B_-^{n+1-k}(\Rmax) \to \Om_i$ for some 
$\Rmax > 0$ and for $i=1,2$. We have $W_i' = w_i(W \times \{ 0 \})$. We obtain a new manifold with boundary $\Om^\#$ by gluing $\Om_1 \setminus (w_1 ( W \times B_-^{n+1-k}(\Rmax)))$ and $\Om_2 \setminus (w_2(W \times B_-^{n+1-k}(\Rmax)))$ 	along $w_i(W	\times S^{n-k})$. 
This manifold is again called {\it the connected sum of $\Om_1$ and $\Om_2$ along $W$}. Let $M^\# \definedas \partial \Om$. Then, $M^\#$ is the connected sum of $M_1$ and $M_2$ along $W$ as explained above. 

\noindent In the special case that $(\Om_2,g_2)$ is the half-sphere $\mS_+^{n+1}$ (and hence $M_2$ is the standard $n$-dimensional sphere) and that $W = \mS^k \subset \mS^n = \partial \mS^{n+1}_+$, then one can verify that the resulting manifold $\Om^\#$ is obtained from $\Om_1$  by attachment of a $(k+1)$-dimensional handle as explained in paragraph \ref{surgatt} and hence, $M^\#$ is obtained from $M_1$ by a surgery of dimension $k$. 
%\{check}

\noindent In what follows, we assume that the metrics $g_i$ have a product form $h_i+ds_i^2$ near the boundaries $M_i$. We define the disjoint union
$$ 
(\Om,g) \definedas (\Om_1 \amalg \Om_2, g_1 \amalg g_2),
$$
$$(M,h) \definedas  (M_1 \amalg M_2, h_1 \amalg h_2)$$
and 
$$
W' \definedas W_1' \amalg W_2'.
$$
Let $r_i$ be the function on $\Om_i$ giving the distance to $W_i'$ associated to the metric $g_i$. Since the metric $g_i$ has the product form $h_i+ds_i^2$ near $M_i$, we have 
\begin{eqnarray} \label{pythagore}
r_i^2 = s_i^2 + (\dist_{h_i}(\cdot,W_i'))^2.
\end{eqnarray}
 We also have $r_1 \circ w_1 (p,x) = r_2 \circ w_2(p,x) = |x|$ for $p \in W$,
$x \in B^{n+1-k}_-(\Rmax)$. Let $r$ be the function on $M$ defined by
$r(x) \definedas r_i(x)$ for $x \in M_i$, $i=1,2$. For $\ep >0$ we
set $U_i(\ep) \definedas \{ x \in M_i \, : \, r_i(x) < \ep \}$ and
$U(\ep) \definedas U_1(\ep) \cup U_2(\ep)$. For $0 < \ep < \th$ we
define
$$
\Om_{\ep}^\# 
\definedas
( \Om_1 \setminus U_1(\ep) ) \cup ( \Om_2 \setminus U_2(\ep) )/ {\sim},
$$
and 
$$
U^{\Om^\#}_\ep (\th)
\definedas
(U(\th) \setminus U(\ep)) / {\sim}
$$
where ${\sim}$ indicates that we identify $x \in \partial U_1(\ep)$
with $w_2 \circ w_1^{-1} (x) \in \partial U_2(\ep)$. Hence
$$
\Om_{\ep}^\# 
=
(\Om \setminus U(\th) ) \cup U^{\Om^\#}_\ep (\th).
$$

\noindent We say that $\Om_\ep^\#$ is obtained from $M_1$, $M_2$ (and $\bar{w}_1$, 
$\bar{w}_2$) {\it by a connected sum along $W$ with parameter $\ep$}. 

\noindent The diffeomorphism type of $\Om_\ep^\#$ is independent of $\ep$, hence unless when the parameter $\ep$ is needed, we
will usually write $\Om^\# = \Om_\ep^\#$.
% and  
%$U^{\Om^\#}(\th) =U^{\Om^\#}_\ep (\th)$.

%\noindent The surgery operation on a manifold is a special case of taking
%connected sum along a submanifold. Indeed, let $M$ be a compact
%%manifold of dimension $n$ and let $M_1 = M$, $M_2 = S^n$, 
%%$W = S^k$. Let $w_1 : S^k \times B^{n-k} \to M$ be an embedding
%defining a surgery and let $w_2 :  S^k \times B^{n-k} \to S^n$ be the
%standard embedding. Since $S^n \setminus w_2 (S^k \times B^{n-k})$ is
%diffeomorphic to $B^{k+1} \times S^{n-k-1}$ we have in this situation
%that $N$ is obtained from $M$ using surgery on $w_1$, see 
%\cite[Section VI, 9]{kosinski:93}. 

%%%%%%%%%%%%%%%%%%%%%%%%%%%%%%%%%%%%%%%%%%%%%%%%%%%%%%%%%%%%%%%%%%%%%%%%%%%%%%%%%%%%%%%%%%%%%%%%%%%
\subsection{Surgery and Yamabe invariants}
\subsubsection{Statement of the results}
First, we will need the following theorem due to Gromov-Lawson and
Schoen-Yau (\cite{gromov.lawson:80} and
\cite{schoen.yau:79}) and which can be also deduced from the
\cite{ammann.dahl.humbert:08}:  
% \begin{theorem} \label{main_surgery_interior}
% Let $n\in \mN$, $n \geq 3$ and $(\Om_1,g_1)$, $(\Om_2,g_2)$ be compact $(n+1)$-dimensional  
% Riemannian manifolds with respective minimal boundaries $M_1$ and
% $M_2$. Set $h_i =\partial g_1$ and define $\Om = \Om_1 \amalg \Om_2$. 
% Let also $W$ be a a compact closed manifold of dimension $k \in \{0,\cdots,n-2 \}$ that embedds in $\stackrel{\circ}{\Om}$ (see Paragraph \ref{connected_sum}).  
% Let $\Om^\#$ be the connected sum of $\Om_1$ and $\Om_2$ along $W$. 
% Then, there exists some constants $\alpha_{n,k} >0$ depending only on $n$
% and $k$, with $\al_{n,0} = + \infty$ and a sequence of metrics $(g_\th)_{\th>0}$  on $\Om^\#$ equal to $g= g_1 \amalg g_2$ except in a small neigbourhood of $W$ (if $U$ is a small neighborhood of $W$, we see $W\setminus U$ as embedded in $\Om^\#$) such   
% that $\partial g_\th = h_1 \amalg h_2$, 
% \begin{eqnarray}
% \lim_{\th \to 0} \mu(\Om^\#, [g_\th];  M^\#,[h_1]  \amalg [h_2]) & \geq&  
% \min( \mu(\Om_1 \amalg \Om_2, [g_1] \amalg [g_2] ; 
% & = & \min(\mu(\Om_1,[g_1];M_1,[h_1]),\mu(\Om_2,[g_2]; M_2,[h_2]), \al_{n,k}).
% \end{eqnarray} 
% \end{theorem}

% \noindent The proof of Theorem \ref{main_surgery_interior} is exactly the
% same than the one of Theorem 2.3.  in \cite{ammann.dahl.humbert:08}.

% \noindent As a corollary, this gives a new proof of a well known theorem of Gromov-Lawson:

\begin{theorem} \label{main_surgery_interior}
Let $(\Om,G)$ be a compact Riemannian manifold of dimension greater than $3$ with boundary $M$ and  let
$\Om^\#$ be obtained from $M$ by an interior surgery of codimension at
least $3$. Assume that $\mu(\Om, [G]; M, \partial [G])>0$. Then, there exists on $\Om^\#$ a
metric $G^\#$ equal to $G$ in a neighborhood of $M = \partial \Om =
\partial \Om^\#$ such that $\mu(\Om^\#, [G^\#]; M, \partial [G])>0$.     
\end{theorem} 

\noindent Let us deal now with the case where $W$ embedds in the
boundary. We prove the following result which in view of Paragraph
\ref{connected_sum} is stronger than Theorem \ref{main}.  
\begin{theorem} \label{main_surgery}
Let $n\in \mN$, $n \geq 2$ and $(\Om_1,g_1)$, $(\Om_2,g_2)$ be compact $(n+1)$-dimensional  
Riemannian manifolds with respective boundaries $M_1$ and $M_2$ and set
$h_i =\partial g_i$. 
Let also $W$ be a a compact manifold without boundary of dimension $k \in \{0,\cdots,n-2 \}$ that embedds in $M_i$ (see paragraph \ref{connected_sum}).  
Let $\Om^\#$ be the connected sum of $\Om_1$ and $\Om_2$ along $W$ and set $M^\# \definedas \partial \Om^\#$.
Then, there exists some constants $\alpha_{n,k} >0$ depending only on $n$
and $k$ with $\al_{n,0} = + \infty$ and a sequence of metrics
$(g_\th)_{\th>0}$  on $\Om^\#$ equal to $g= g_1 \amalg g_2$ 
except in a small neigbourhood of $W$ (if $U$ is a small neighborhood of $W$, we see $W\setminus U$ as embedded in $\Om^\#$) such   
that, if we note $h_\th=\partial g_\th$, 
\begin{eqnarray} \label{result_on_Y}
\lim_{\th \to 0} \mu(\Om^\#, [g_\th]; M^\#,[h_\th]) \geq \min(\mu(\Om_1,g_1;M_1,[h_1]),\mu(\Om_2,[g_2];M_2,[h_2]), \al_{n,k}).
\end{eqnarray}
If in addition, $n \geq 3$ and $k \leq n-3$, the metrics $h_\th \definedas \partial g_\th$
coincides with the metrics given by Theorem 2.3 in \cite{ammann.dahl.humbert:08}. In
other words, there exists a constant $\beta_{n,k}>0$ (the same as in
Theorem 
\ref{adh}) with $\beta_{n,0} = +
\infty$  such that  
\begin{eqnarray} \label{result_on_mu}
\lim_{\th \to 0} \mu(M^\#,[\partial g_\th]) \geq \min(\mu(M_1, [h_1]), \mu(M_2, [h_2]), \beta_{n_k}.
\end{eqnarray} 
Moreover, for $k=0$, we have 
\begin{eqnarray*} 
\alpha_{n,0} = \beta_{n,0} = + \infty.
\end{eqnarray*}
\end{theorem}

%%%%%%%%%%%%%%%%%%%%%%%%%%%%%%%%%%%%%%%%%%%%%%%%%%%%%%%%%%%%%%%%%%%%%%%%%%%%%%%%%%%%%%%%%%%%%%%
\subsubsection{Proof of Theorem \ref{main_surgery}} 
We use the notations of Paragraph \ref{connected_sum}. We recall the
notations  $\Om=\Om_1 \amalg \Om_2$, $M = M_1 \amalg M_2$, $W'= W_1' \amalg
W_2'$ and $g = g_1 \amalg g_2$. We also use the notation $h \definedas
\partial g$. If $(g_m)$ is a sequence of metric which converges toward a
metric $g_\infty$ in $C^0(\Om)$ and if $\Scal_{g_m}$ converges also in
$C^0$ to $\Scal_{g_\infty}$ then $\mu(\Om,[g_m];\partial \Om,[\partial
g_m]))$ tends to $\mu(\Om,[g_\infty]; \partial \Om, [\partial g_\infty])$
(see Proposition 4.31 of B\'erard-Bergery in \cite{besse:87} and Lemma 4.1 in \cite{akutagawa.botvinnik:02a}).
Theorem 4.6 in \cite{akutagawa.botvinnik:02a} or the results of Carr
\cite{carr:88} then imply that we can choose a metric $\tilde{g}$ on $\Om$ such that:
\begin{itemize}
\item $\partial \tilde{g}= \partial{g}= h$, 
\item $\tilde{g} = h + ds^2$ in a neighborhood of $M$ (where $s=s_i$ on $M_i$ with $s_i$ defined as in the end of Paragraph \ref{connected_sum}),  
\item $\mu(\Om,[\tilde{g}]; M, [h])$ is as close as desired to $\mu(\Om,[g];M,[h])=$ \\
$\min(\mu(\Om_1,[g_1]; M_1,[h_1]),\mu(\Om_2,[g_2]; M_2,[h_2]))$. 
\end{itemize}
Then, without loss of generality, we can replace $g$ by $\tilde{g}$ so that the metric has now the above properties.  
The desired sequence $(g_\th)$ of metrics will be constructed as in \cite{ammann.dahl.humbert:08}. We now explain how this construction can be adapted here. In the
following, $C$ denotes a constant that might change its value between
lines. We denote by $h_i'$ the restriction of $g_i$ to $TW'= T(W_1' \amalg W_2')$ over $W' \subset \Om$. 
As already explained, the normal exponential map of $W'$
defines a diffeomorphism 
$$
w_i: W \times B_-^{n+1-k}(\Rmax)
\to 
U_i(\Rmax),\quad i=1,2,
$$
which decomposes $U(\Rmax) = U_1(\Rmax) \amalg U_2(\Rmax)$ as a
product $W' \times B_-^{n+1-k}(\Rmax)$. In general the Riemannian metric
$g$ does not have a corresponding product structure, and we introduce
an error term $T$ measuring the difference from the product metric. If
$r$ denotes the distance function to $W'$, then the metric $g$ can be
written on 
$U(\Rmax)\setminus W'\cong W'\times (0,\Rmax)\times S_-^{n-k}$ as
\begin{equation} \label{metric=product} 
g =  h' + \xi^{n+1-k} + T = h' + dr^2 + r^2 \sigma^{n-k} + T.
\end{equation} 
where $h'$ is the restriction of $g$ on $TW'$, $T$ is a symmetric $(2,0)$-tensor vanishing on $W'$ 
(in the sense of sections of $(T^*\Om \otimes T^*\Om)|_{W'}$). Note that
since $g$ is a product near the boundary, 

\begin{eqnarray} \label{Tnorm}
T\left( v, \cdot\right) = 0
\end{eqnarray}
for all vector $v$ normal to $M$.  
We also define the product metric
\begin{equation} \label{def.g'}
g' \definedas h' + \xi^{n+1-k} = h' + dr^2 + r^2 \si^{n-k},
\end{equation}
on $U(\Rmax) \setminus W'$. Thus $g = g' + T$. 
 We define $T_i \definedas T|_{\Om_i}$ for
$i=1,2$. 

\noindent For a fixed $R_0\in (0,\Rmax)$ we choose a smooth positive function  
$F: \Om \setminus W' \to \mR$ such that 
$$
F(x) = 
\begin{cases}
1, 
&\text{if $x \in \Om_i \setminus U_i(\Rmax)$;} \\ 
r_i(x)^{-1}, 
&\text{if $x \in U_i(R_0)\setminus W'$.}
\end{cases}
$$
Next we choose small numbers $\theta, \detwo \in (0,R_0)$ with 
$\th> \detwo > 0$. Here ``small'' means that we first choose a
sequence $\th=\th_j$ of small positive numbers tending to zero, such
that all following arguments hold for all $\th$.  Then we choose for
any given $\th$ a number $\detwo = \detwo(\th) \in (0,\th)$ such that
all arguments which need $\detwo$ to be small will hold,
%Then, in a third step we choose  $\de_1 = \de_1(\th)\in (0,\detwo)$
% in a similar way, 
%see Figure~\ref{hier_var}.
%%%%%%%%%%%%%%%%%%%%%%%%%%%%%%%%%%%%%%%%%%%%%%%%%%%
%%\begin{figure}
%\begin{center}
%$\framebox{\vbox{
%{\sc\large Hierarchy of parameters} 
%$$
%\Rmax> R_0 > \th> \detwo > \ep>0 
%$$
%We choose parameters in the order  
%$\Rmax,R_0,\th,\detwo, A_{\th}$. 
%We then set 
%$\ep \definedas e^{-A_{\th}}\detwo$.\\
%This implies $|t|=A_{\th}\Leftrightarrow r_i=\detwo$.
%}}$

%\caption{Hierarchy of parameters} \label{hier_var}
%\end{center}
%\end{figure}
%%%%%%%%%%%%%%%%%%%%%%%%%%%%%%%%%%%%%%%%%%%%%%%%%%
%
For any $\th>0$ and sufficiently small $\detwo$ there is 
$A_\th\in [\th^{-1}, (\detwo)^{-1})$ and a smooth function 
$f: U(\Rmax) \to \mR$ depending only on the coordinate $r= \dist_g(\cdot,W')$ such
that 

$$
f(x)  =  
\begin{cases}
         -  \ln r(x),  &\text{if $x \in U(\Rmax) \setminus U(\th)$;} \\
\phantom{-} \ln A_\th, &\text{if $x \in U(\detwo)$,} 
\end{cases}
$$
and such that
\begin{equation} \label{asumpf}
\left| r\frac{df}{dr} \right|
= 
\left| \frac{df }{d(\ln r)} \right|
\leq 1, 
\quad 
\text{and} 
\quad
\left\|r\frac{d}{dr}\left(r\frac{df}{dr}\right)\right\|_{L^\infty}
= 
\left\|\frac{d^2f}{d^2(\ln r)}\right\|_{L^\infty}
\to 0
\end{equation}
as $\th\to 0$. 
We set $\ep = e^{-A_\th} \detwo$. We can and will assume that $\ep<1$.
Let $\Om^\#$ be obtained from~$\Om$ by a connected sum along $W$
with parameter $\ep$, as described in Paragraph~\ref{connected_sum}. 
In particular,
$U^{\Om^\#}_\ep(s) = U(s)\setminus U(\ep)/{\sim}$ for all $s\geq \ep$.
On the set $U^{\Om^\#}_\ep(\Rmax) = U(\Rmax) \setminus U(\ep)/{\sim}$ we define  
the variable $t$ by 
$$
t \definedas
\begin{cases}
         -  \ln r_1 + \ln \ep, & 
\text{on $U_1(\Rmax) \setminus U_1(\ep)$;} \\
\phantom{-} \ln r_2 - \ln \ep, & 
\text{on $U_2(\Rmax) \setminus U_2(\ep)$.}
\end{cases}
$$
%Note that $t \leq 0$ on $U_1(\Rmax) \setminus U_1(\ep)$ and $t \geq 0$
%on $U_2(\Rmax) \setminus U_2(\ep)$, with $t=0$ precisely on the common
%boundary $\pa U_1(\ep)$ identified with $\pa U_2(\ep)$ in $N$. It
%follows that  
%$$
%r_i = e^{|t|+ \ln \ep} = \ep e^{|t|}.
%$$
We can assume that 
$t:U^{\Om^\#}_\ep(\Rmax) \to \mR$  is smooth. 
%Expressed in the variable $t$
%we have 
%$$
%F(x) = \ep^{-1}e^{-|t|}
%$$ 
%for $x \in U(R_0) \setminus U^{\Om^\#}_\ep(\th)$, or in other words if 
%$|t|+\ln \ep \leq \ln R_0$. Then Equation \eref{metric=product} tells
%us that 
%$$
%F^2 g =\ep^{-2} e^{-2|t|}(h+T) + dt^2 + \sigma^{n-k}
%$$
%on $U(R_0)\setminus U^{\Om^\#}_\ep(\th)$. 
%%%%%%%%%%%%%%%%%%%%%%%%%%%%%%%%%%%%%%%%%%%%%%%%%
% The metrics g_\th Double picture
%%%%%%%%%%%%%%%%%%%%%%%%%%%%%%%%%%%%%%%%%%%%%%%
%If we consider $f$ as a function of $t$,
%then 
%$$
%f(t) =  
%\begin{cases}
%-|t|-\ln\ep,  
%&\text{if $\ln\th- \ln \ep \leq |t| \leq \ln\Rmax - \ln\ep$;} \\
%\ln A_\th,  
%&\text{if  $|t|\leq \ln\detwo-\ln\ep$;}
%\end{cases}
%$$
%and $|df/dt|\leq 1$, $\|d^2f/dt^2\|_{L^\infty}\to 0$. 
We choose a
cut-off function $\chi:\mR\to [0,1]$ such that $\chi=0$ on
$(-\infty,-1]$, $|d\chi| \leq 1$, and $\chi=1$ on $[1,\infty)$. With
these choices, we define 
$$
g_{\th}
\definedas  
\begin{cases}
F^2 g_i, 
&\text{on $\Om_i \setminus U_i(\th)$;} \\
e^{2f(t)}(h_i'+T_i) + dt^2 + \sigma^{n-k},
&\text{on $U_i(\th)\setminus U_i(\detwo)$;} \\
\begin{aligned}
&A_{\th}^2 \chi( t / A_{\th} )(h_2'+T_2)
+ A_{\th}^2 (1-\chi( t / A_{\th} ))(h_1'+T_1)\\
&\quad + dt^2 + \sigma^{n-k}, 
\end{aligned}
&\text{on $U^{\Om^\#}_\ep(\detwo)$.} 
\end{cases}
$$
%On $U^N(R_0)$ we write $g_{\th}$ as 
%$$
%%g_\th= e^{2f(t)} \tilde{h}_t + dt^2 + \sigma^{n-k-1} + \widetilde{T}_t,
%%$$
%where the metric $\tilde{h}_t$ is defined for $t \in \mR$ by 
%$$
%\tilde{h}_t 
%%\definedas 
%\chi(t / A_{\th}) h_2 + (1 - \chi(t / A_{\th})) h_1,
%$$
%and where the error term $\widetilde{T}_t$ is equal to 
%$$
%\widetilde{T}_t 
%\definedas 
%e^{2f(t)}
%\left(
%\chi(t/A_{\th})T_2 + \left( 1 - \chi(t/A_{\th}) \right) T_1
%\right).
%$$
%See also Figure~\ref{fig.gth}.
It remains to proves that the sequence $(g_\th)$ satisfies the desired
conclusions. Set $h_\th \definedas \partial g_\th$. First of all, we prove
that $M^\#$ is minimal for the metrics $g_\th$. Let $p \in M^\#$. Assume
first that $p \in \Om_i \setminus U_i(\th)$. Note that the function $F$
depends only on the coordinate $r$. We denote by $\nu$ the outer normal
unit vector at $p$. 
Formula \eref{pythagore} then implies that $\partial_\nu r \equiv 0$ on $M
\setminus W'$ and hence $\partial_\nu F(p)=0$. This implies that the mean
curvature vanishes at $p$. Assume now that $p \in U_i(\th)\setminus
U_i(\detwo) \cup U^{\Om^\#}_\ep(\detwo)= U^{\Om^\#}_\ep(\th)$.
Observe that by Relation \eref{Tnorm}, the metric $g_\th$ has the form 
$$g_\th = H_1+ \al(r)H_2 + dt^2 + \sigma^{n-k}$$ 
where $H_i$ are 2-forms satisfying $H_i(\nu,\cdot) \equiv 0$ and  where
$\alpha(r)$ is a function depending only on $r$ and $\th$. 
Set $r_\th \definedas \dist_{h_\th}(W',\cdot)$. Then, $dt^2 =
\frac{1}{r^2} dr^2  = \frac{1}{r^2}(r_\th^2 +ds^2)$. Since $\partial_\nu r
\equiv 0$, we easily get that the mean curvature vanishes at $p$ and hence
$M^{\#}$ is minimal. \\

\noindent Assume for a while that $k \leq n-3$. Observe that since $g$ is a
product metric near $M$, the function $r=\dist_g(\cdot,W')$ coincides with
$\dist_h(\cdot,W')$ on the boundary. Consequently, the metric $h_\th$ on
$M^\#$ is exactly the same than the one constructed in the proof of Theorem
2.3 of \cite{ammann.dahl.humbert:08} which then shows that Relation \eref{result_on_mu} holds.

\noindent Let us come back to the general case $k \in \{ 0,\cdots,n-2\}$ and let us show Relation \eref{result_on_Y}. 
Let us denote by $X_i \definedas \Om_i \cup_{M_i} \Om_i$ ($i =1,2$)
(resp. $X^\# \definedas \Om^\# \cup_{M^\#} \Om^\#$) the double of $\Om_i$
(resp. $\Om^\#$). Notice that $X^\#$ is the connected sum of  $X_1$
and $X_2$ along $W$. We define on $X_i$ (resp. $X^\#$)  the metric
$\bar{g_i}= g_i \cup g_i$ (resp. $\bar{g}_\th=g_\th \cup g_\th$) as in
Paragraph~\ref{double}. Set also $X \definedas X_1 \amalg X_2$ and $\bar{g} \definedas
\bar{g}_1 \amalg \bar{g}_2$. The manifold $X$ is then the double of $\Om$. 
 Clearly, we can assume that
$\mu(\Om^\#,[g_\th]; M^\#,[h_\th]) < \mu(\mS^{n+1}_+,\mS^{n})$ otherwise the
proof is done.
By Theorem \ref{escobar}, there exists a function $u_\th \in C^\infty(\Om^\#)$, $u_\th >0$  normalized by 
$$\int_{\Om^\#} u_\th^\frac{2(n+1)}{n-1} dv_{g_\th}= 1 $$
which satisfies 
\[\left\{ \begin{array}{ccc} 
L_{g_\th} u_\th =  \la_\th u_\th^{\frac{n+3}{n-1}}  & \hbox{ on } &
\stackrel{\circ}{(\Om^\# )}\\
\partial_\nu u_\th = 0 & \hbox{ on } & M^\#,
\end{array} \right. \]
where $\la_\th=  \mu(\Om^\#, [g_\th];M^\#,[h_\th])$. 
By possibly taking a subsequence, we can assume that $\la_\infty \definedas \lim_{\th \to 0} \la_\th \in [-\infty, \mu(\mS^{n+1}_+,\mS^{n})]$ exists. 

\noindent Define 
$$\bar{u}_\th \definedas \frac{u_\th \cup u_\th}{2^\frac{n-1}{2(n+1)}} =   \frac{u_\th \cup u_\th}{\|u_\th \cup u_\th\|_{L^\frac{2(n+1)}{n-1}
(X^\#)}} $$  
on $X^\#$. 
Then, 
$$\int_{X^\#} \bar{u}_\th^\frac{2(n+1)}{n-1} dv_{\bar{g}_\th}= 1.$$
Proposition \ref{basic} implies that $\bar{u}_\th$ is smooth on $X^\#$ and satisfies 
$$L_{\bar{g}_\th} \bar{u}_\th = 2^{-\frac{n-1}{2(n+1)}} \la_\th (u_\th \cup
u_\th)^{\frac{n+3}{n-1}} = 2^\frac2{n+1} \la_\th \bar{u}_\th^\frac{n+3}{n-1}.$$
The idea now is to see how the proof of Theorem 2.3 in 
\cite{ammann.dahl.humbert:08} can be adapted to this situation. The first observation
is that the metric $(X^{\#},\bar{g}_\th)$ is construted from $(X,\bar{g})$
exactly in the same way than $(N,g_\th)$ is constructed from $(M,g)$ in
\cite{ammann.dahl.humbert:08}. We deduce immediatly that 
$$2^{\frac{2}{n+1}} \la_\infty \geq \min( \mu(X,[\bar{g}]), \beta_{n+1,k})$$
where $\be_{n,k}$ is as in Theorem \ref{adh}. The problem here is to get a
lower bound of $\mu(X,[\bar{g}])$ in terms of $\mu(\Om,[g];
M,[h])$ which seems difficult  without additional assumptions. So we have
to go through the proof in \cite{ammann.dahl.humbert:08} a little more
deeply. Observe that it is divided in many cases. The only case which is an
issue is Subcase II.1.2. Indeed, in other cases, we obtain that
$2^{\frac{2}{n+1}} \la_\infty \geq  \beta_{n+1,k}$ and we just set 
$\alpha_{n,k} \definedas 2^{- \frac{2}{n+1}}\beta_{n+1,k}$ to get Theorem
  \ref{main_surgery}. So assume now that assumptions of Subcase II.1.2
  occur. 
More precisely, we assume, using the notations of Paragraph
\eref{connected_sum}  that: 

$$\lim_{b\to 0} \limsup_{\th\to 0} \sup_{U^{X^\#}_\ep (b)}
  \bar{u} _\th= 0$$
where 
$$U^{X^\#}_\ep(b)\definedas U^{\Om^\#}_\ep(b) \cup_{\partial U^{\Om^\#}_\ep(b) \cap \partial \Om^\#} U^{\Om^\#}_\ep(b).$$
\noindent We then mimick the proof of \cite{ammann.dahl.humbert:08}. Let $d_0 >0$. We can choose a $b >0$ such that 

$$\int_{X^\# \setminus U^{X^\#}_\ep(2b)} \bar{u}_\th^{\frac{2(n+1)}{n-1}} dv_{\bar{g}_\th} \geq 1-d_0$$ 
and 
$$\int_{U^{X^\#}_\ep(2b) \setminus U^{X^\#}_\ep(b)} \bar{u}_\th^2 dv_{\bar{g}_\th} \leq d_0.$$ 
Then we choose a cut-off function $\eta\in C^{\infty}(X^\sharp)$, $0 \leq \eta \leq 1$  depending only on $t$ (clearly the function $t$ can be naturally extended smoothly to $X^\#$) equal to $0$ on $U^{X^\#}_\ep(b)$, equal to $1$ on $X^\# \setminus U^{X^\#}_\ep(2b)$ and which satisfies 
$|d \xi|_{\bar{g}_\th} \leq 2 \ln(2)$. 
Then, as in \cite{ammann.dahl.humbert:08}, we obtain that  
$$J^{n+1}_{X,\bar{g}} (\chi \bar{u}_\th) \leq \frac{2^\frac{2}{n+1} \la_\th + |2^\frac{2}{n+1} \la_\th|d_0 + 4 (\ln(2))^2a d_0}{(1-d_0)^\frac{n-1}{n+1}}$$ 
where $a= \frac{4n}{n-1}$. 
Since $\chi$ depends only on $t$ and hence of $r$,
observe that the function $\chi \bar{u}_\th$ 
has normal derivative vanishing on the minimal hypersurface $M^\# \subset X^\#$. 
By Proposition \ref{basic2}, we obtain that 
$$2^\frac{2}{n+1} \mu(\Om,[g];M,[h]) \leq   J^{n+1}_{X,\bar{g}} (\chi \bar{u}_\th)$$ and hence, letting $d_0$ tends to zero,
$$\mu(\Om,[g];M,[h]) \leq \la_\infty.$$
This proves Theorem \ref{main_surgery}.
%%%%%%%%%%%%%%%%%%%%%%%%%%%%%%%%%%%%%%%%%%%%%%%%%%%%%%%%%%%%%%%%%%%%%%
\subsection{Surgery on cylinders} \label{surgery_cyl}
\subsubsection{Statements of the results}

\noindent Let $M,N$ be a compact $n$-dimensional manifold without boundary. Assume that $N$ is obtained from $M$ by a surgery of dimension $k \in \{0,\cdots,n-1 \}$ associated to an embedding $f:S^k \times B^{n-k} \hookrightarrow M$.  Let $\Om = M \times [0,1]$. Attaching on $\Om$ two $(k+1)$-dimensional handles along 
$f(	S^k \times B^{n-k}) \times\{0,1\}$, we get a new manifold $\Om'$ whose boundary is $N \amalg N$ (see Paragraph \ref{surgery}).
We prove: 

\begin{lemma} \label{surgery_cyl_1}
The manifold $N \times [0,1]$ is obtained from $\Om'$ by an interior $(k+1)$-dimensional surgery.
\end{lemma}

\noindent Start again with $\Om=M \times [0,1]$. Let $\Om''$ be obtained
from attaching first a $k$-dimensional handle on $\Om$ along $f(S^k \times
B^{n-K}) \times \{ 0 \}$ and then attaching the dual handle along $N
\subset \left( \Om \cup_{ f(S^k \times B^{n-K}) \times \{ 0 \}} \overline{B^{k+1}}
  \times\overline{ B^{n-k}} \right)$. The new manifold $\Om''$ has a boundary 
$M \amalg M$. We prove 

\begin{lemma} \label{surgery_cyl_2}
The manifold $M \times [0,1]$ is obtained from $\Om''$ by an interior $(n-k)$-dimensional surgery.
\end{lemma}

\begin{remark} \label{0surgery}
If $k= 0$ Lemma \ref{surgery_cyl_2}, then by standard surgery theory, the
interior $n$-dimensional surgery can be replaced by an interior surgery of
dimension $1$.
\end{remark}
%%%%%%%%%%%%%%%%%%%%%%%%%%%%%%%%%%%%%%%%%%%%%%%%%%%%%%%%%%%%%%%%%%%%%%%%%%%%%%%%%%%%%%%%%%%%%%%%
\subsubsection{Proof of Lemma \ref{surgery_cyl_1}}
The manifold $\Om'$ is equal to 
$$\Om' = \left( B^{k+1}  \times B^{n-k} \right)  \cup_{f(S^k \times B^{n-k}) \times \{ 0 \}} \Om \cup_{f(S^k \times B^{n-k}) \times \{ 1 \}} \left( B^{k+1}  \times B^{n-k} \right).$$
We define
\begin{eqnarray*} 
W & \definedas  & \left(B^{k+1}  \times B^{n-k}\left(\frac{1}2 \right)  \right) 
 \cup_{f(S^k \times B^{n-k}\left(\frac{1}2 \right) ) \times \{ 0 \}} \\
& &    \left( f\left(S^k \times B^{n-k}\left(\frac{1}2 \right) \right) \times [0,1] \right) \cup_{f\left(S^k \times B^{n-k}\left(\frac{1}2 \right) \right)
 \times \{ 1 \}} \\
 & &\left( B^{k+1}  \times B^{n-k} \left(\frac{1}2 \right) \right) \subset\; \stackrel{\circ}{(\Om')}.
\end{eqnarray*}

\noindent Let $m \in \mN$. Observe that 

\begin{eqnarray}
B^{m+1} \cup_{S^m \times \{ 0\} } \left( S^{m} \times [0,1] \right) \cup_{S^m \times \{1\} } B^{m+1} \simeq S^{m+1}.
\end{eqnarray}

\noindent Here, $\simeq$ means diffeomorphic. Hence,
$W \simeq S^{k+1} \times B^{n-k}$. 

\noindent Define 
\begin{eqnarray*}
W' & \definedas & \left(B_+^{k+2} \times S^{n-k-1} \right) \cup_{B^{k+1} \times S^{n-k-1} \times\{ 0\} } \\
& & \left( B^{k+1} \times S^{n-k-1} \times [0,1] \right) \cup_{S^{k+1} \times S^{n-k-1} \times\{ 1\} } \\
&& \left(B_+^{k+2} \times S^{n-k-1} \right).
\end{eqnarray*}

\noindent Note that $\partial B_+^{k+2}= S_+^{k+1} \cup_{S_k} B^{k+1}$ hence $W'$ is well defined. For $m \in \mN$, let us note that  
$$B_+^{m+1}  \cup_{B^m \times \{ 0 \} }  \left( B^m \times [0,1] \right) \cup_{B^m \times \{ 1 \} } B_+^{m+1} \simeq B^{m+1}.$$

\noindent Hence, $W' \simeq B^{k+2} \times S^{n-k-1}$ and if we define 
$$\Om^\# \definedas (\Om' \setminus W) \cup W'$$
where we glue the boundaries, $\Om^\#$ is obtained from
$\Om'$ by an interior $(k+1)$-dimensional surgery along $W$. 

\noindent Define 

\begin{eqnarray*} 
H & \definedas & \left( B^{k+1} \times B^{n-k} \right) \setminus \left( B^{k+1} \times B^{n-k}\left(\frac{1}2 \right)  \right) 
\cup_{B^{k+1} \times S^{n-k-1}\left(\frac{1}2 \right)  \simeq S^{k+1}_+ \times S^{n-k-1}} \\
&  & \left( B^{k+2}_+ \times S^{n-k-1} \right) \\
&  \simeq & \left( B^{k+1} \times S^{n-k-1}\left(\frac{1}2 \right) \times [\frac{1}{2},1] \right)  \cup_{B^{k+1} \times S^{n-k-1} \simeq S^{k+1}_+ \times S^{n-k-1}} \\
& & \left( B^{k+2}_+ \times S^{n-k-1} \right). 
\end{eqnarray*}

\noindent Since
$$\left( B^{k+1}  \times [\frac{1}{2},1] \right) \cup_{B^{k+1} \times \{ \frac{1}{2} \}  \simeq S^{k+1}_+} B_+^{k+2} \simeq B^{k+2}_+$$
we see that
$$H \simeq B^{k+2}_+ \times S^{n-k-1}.$$
\noindent Now observe that 
$$\Om^\# = H \cup_{B^{k+1} \times S^{n-k-1} \times \{ 0 \}}  \left( N \times [0,1] \right) \cup_{B^{k+1} \times S^{n-k-1} \times \{ 1 \} } H$$ 
\noindent It is not difficult to see that  $\Om^\# \simeq N \times [0,1]$. This proves Lemma \ref{surgery_cyl_1}.

%%%%%%%%%%%%%%%%%%%%%%%%%%%%%%%%%%%%%%%%%%%%%%%%%%%%%%%%%%%%%%%%%%%%%%%%%%%%%%%%%%%%%%%%%%%%%%%%%
\subsubsection{Proof of Lemma \ref{surgery_cyl_2}}
\noindent Let 
$$H \definedas \left(B^{k+1} \times B^{n-k}\right) \cup_{B^{k+1} \times S^{n-k-1}} \ \left(B^{k+1} \times B^{n-k} \right).$$  
\noindent We have 
$$\partial H = \left(S^k \times B^{n-k}\right) \cup_{S^k \times S^{n-k-1}}  \left(S^k \times B^{n-k}\right). $$
Since for all $m \in \mN$, $n \geq 1$, 
$$B^m \cup_{S^{m-1}} B_m \simeq S^m$$
(by smoothing the corners), we have 

$$H \simeq B^{k+1} \times S^{n-k} \; \hbox{ and } \; \partial H 
\simeq S^k \times S^{n-k}.$$ 

\noindent By construction, $\Om''$ is equal to 
 $$ \Om'' =  \Om \cup_{f(S^k \times B^{n-k}) } H.$$

\noindent Now, we set 

\begin{eqnarray*} 
W & \definedas & \left(B^{k+1}\left(\frac{1}{2}\right) \times B^{n-k}\right) \cup_{B^{k+1}\left(\frac{1}{2}\right) \times S^{n-k-1}} \left(B^{k+1} \left(\frac{1}{2}\right) \times B^{n-k} \right)\\
&  \simeq &  B^{k+1}\left(\frac{1}{2}\right)  \times S^{n-k} \subset \; \stackrel{\circ}{H}.
\end{eqnarray*}
We now perform a surgery on $\Om''$ along $W$ to get a new manifold $\Om^\#$.  Then, 
\begin{eqnarray} \label{omsharp}
\Om^\# = \Om \cup_{f(S^k \times B^{n-k}) } H^\#
\end{eqnarray}
where 
\begin{eqnarray*} 
H^\#&  \simeq  & \left(B^{k+1}  \times B^{n-k}\right) \setminus \left(B^{k+1}\left(\frac{1}{2}\right) \times B^{n-k}\right) \cup_{S^k
 \times S^{n-k} } \left(S^k  \times B^{n-k+1}\right) \\
& \simeq & \left( \left[ \frac{1}{2}, 1 \right] \times S^k \times S^{n-k} \right) \cup_{S^k \times S^{n-k}}  \left(S^k  \times B^{n-k+1}\right)\\
& \simeq & S^k \times B^{n-k-1}. 
\end{eqnarray*}

\noindent Note again that 
$$\partial H^\# = \left(S^k \times B^{n-k}\right) \cup_{S^k \times S^{n-k-1}}  \left(S^k \times B^{n-k}\right)$$
and the gluing in Formula \eref{omsharp} is along the first  $\left(S^k \times B^{n-k}\right)$. 
\noindent Now, it is easy to see that $\Om^\# \simeq \Om$. This ends the proof of Lemma \ref{surgery_cyl_2}. 

\section{$c$-concordant metrics} \label{concordant}
%The goal of this section is to prove Theorem \ref{application}. 
Let $M$ be a compact manifold without boundary of dimension $n \geq 3$. Let
$\riem(M)$ be the set of all Riemannian metrics on $M$. For all $c \in \mR$, 
we set 
$$\riem_c(M) = \left\{ g \in \riem(M) \big| \mu(M,[g]) >c \right\}.$$

\noindent Let $g,h$ be Riemannian metrics on $M$ and $c \geq 0 \mR$. We say
that {\it $h,g$ are c-concordant} if $\mu(M \times [0,1]; M \amalg M, [g]
\amalg [h]) >0$ and if $g,h \in \riem_c(M)$. If $M$ is oriented and if $g,h$ have positive scalar
curvature, then $g,h$ are c-concordant if and only if $g,h \in \riem_c(M)$
and  there exists a metric $G$ with positive scalar curvature on $M
\times [0,1]$ such that the boundary is minimal (see
Corollary D in  \cite{akutagawa.botvinnik:02b}).  A consequence of
Theorem 5.1 in  \cite{akutagawa.botvinnik:02a} is the fact that "to be
$c$-concordant" is an equivalence relation. We denote by $\conc_c(M)$ the
set of equivalence classes of concordant metrics. 
For a metric $h$ on a manifold $P$, we denote by $[h]^c_P$ its class in
$\conc_c(P)$. If $c,c' \in \mR$ are such that $c \leq c'$ and if $h \in
\riem_{c'}(M) \subset \riem_{c}(M)$, then we clearly have 
\begin{eqnarray} \label{cc'}
[h]_M^{c'} = [h]_M^c \cap \riem_{c'}(M).  
\end{eqnarray}
 Let $g,h$ be Riemannian metrics in $M$. An important
well-know fact is the following 
\begin{eqnarray} \label{isotopic_implies_conc}
g,h \hbox{ are in the same connected component of } \riem_0(M)
\Longrightarrow [g]^0_M =[h]^0_M.
\end{eqnarray}

\noindent Lots of works aim to study the sets $\riem_0(M)$ and $\conc_0(M)$
(\cite{carr:88,hajduk:88,hajduk:91,rosenberg.stolz:98,ruberman:02}). In particular, Gajer  proved in \cite{gajer:93} very interesting
results about the topology and the structures of these sets.
The reader may also consult Dahl \cite{dahl:06} for a nice study of the set
of metrics with invertible Dirac operator on spin manifolds. \\

\noindent The goal of this section is to show how Theorem \ref{main} can be
applied to collect informations on $\conc_c(M)$  and in particular to prove
Theorem \ref{application}. For this,
we need to introduce lds-relative manifolds. 

% \subsection{Connected components of  $\riem_c(M)$} 
% \subsubsection{Statement of the results} 
% \noindent Let $g,h$ be in the same connected component of $\riem_c(M)$ for some $c
% \geq 0$. Then, they
% are in the same connected component of $\riem_0(M)$ (ce also say they are
% {\it isotopic}) and hence $0$-concordant by a well known result (see
% \cite{}). Since, $g,h \in \riem_c(M)$, they are  $c$-concordant. In
% particular, this proves the following  result  
% \begin{eqnarray} \label{isot_conc}
% \# \conc_c(M) \geq 2 \Longrightarrow \riem_c(M) \hbox{ is not connected.}
% \end{eqnarray}

% \noindent This type of arguments have been used to show that the topology of
% $\riem_0(M)$ is not trivial (\cite{gajer}). We now prove the following
% theorem 

% \begin{theorem} \label{dahl_equivalent}
% \end{theorem}

% \noindent The proof is not difficult : we just collect
% an argument of Dahl
% \cite{} togother with Theorem \ref{main}.

% \subsection{Proof of Theorem \ref{dahl_equivalent}}. 

% %%%%%%%%%%%%%%%%%%%%%%%%%%%%%%%%%%%%%%%%%%%%%%%%%%%%%%%%%%%%%%%%%%%%%%%%%%%%%%
% \subsection{Lds-relative manifolds}
% %%%%%%%%%%%%%%%%%%%%%%%%%%%%%%%%%%%%%%%%%%%%%%%%%%%%%%%%%%%%%%%%%%%%%%%%%%%%%%%
% \subsection{Definitions and properties} 
\begin{defn}
Let $M_1,M_2$ be $n$-dimensional compact manifolds without boundary. We say that $M_1,M_2$ are {\it lds-relative} ("lds" for "low dimensional surgery") if $M_2$  can be obtained from $M_1$ with a finite sequence of surgeries of dimension $2 \leq k \leq n-3$.
\end{defn} 

\noindent \begin{remark} 
\begin{enumerate} 
\item Remark \ref{surgery_cancel} obviously implies that "to be lds-relative" is an equivalence relation.
 We denote by $\lds$ the set of equivalence classes of lds-relative $n$-manifolds.
\item Let $M,N$ be two compact connected $n$-manifolds. Assume that there is a
$2$-connected bordism between $M$ and $N$. Then, it follows from standard
theory that $M,N$ are lds-relative.
\end{enumerate} 
\end{remark}

\noindent An immediate consequence of Theorem \ref{adh} is the
following.

\begin{prop}
Let $\beta_n >0$ be the positive constant defined as in Corollary 
\ref{main_cor}. For all compact $n$-manifold without boundary $M$, we define
$\bar{\sigma}(M) = \min (\sigma(M),\beta_n)$. Then, 
$$\bar{\sigma} : \lds \; \longrightarrow \; ]-\infty, n(n-1) \om_n^{\frac{2}{n}}]$$
where $\om_n$ denotes the volume of the standard $n$-dimensional sphere, is
a well-defined map. 
\end{prop}

\noindent As an application of Theorem \ref{main}, we prove:
\begin{prop} \label{bij}
Let $M, N$ be lds-relative $n$-manifolds. For all $c \leq \beta_n$
($\beta_n$ is as above), there are bijective maps    
$$\Theta^c_{M,N} : \conc_c(M) \to \conc_c(N)$$ such that $\Theta_{M,N}^C =
(\Theta_{N,M}^c)^{-1}$.  
In addition, let $c,c' \in \mR$ with  $c \leq c'$ and let $h \in \riem_{c'}(M)
\subset \riem_c(M)$. Then, 
\begin{eqnarray} \label{s3}
\Theta^{c'}_{M,N}([h]^{c'}_M) = \Theta^{c}_{M,N} ([h]^{c}_N) \cap
\riem_{c'}(N).\\
\nonumber
\end{eqnarray}
\end{prop}

\begin{remark} \label{0surgery2}
Let $M,N$ be compact $n$-manifolds without boundary and assume that $N$ is obtained
from $M$ by a surgery of dimension $0$. In particular, these condition are
satisfied if $M = M_1 \amalg M_2$ and if $N = M_1 \# M_2$ is the connected
sum of $M_1$ and $M_2$. One can verify that the proof of Proposition \ref{bij}
can be mimicked, unless we use Remark \ref{0surgery} instead of Lemma
\ref{surgery_cyl_2} to obtain for all $c$ an injective map  $\Theta^c_{M,N}
: \conc_c(M) \to \conc_c(N)$.\\ 
\end{remark}
% \noindent Let $(M,g)$, $(N,h)$ be two compact Riemannian $n$-manifolds and
% $c \in \mR$. We
% say that $(M,g)$ and $(N,h)$ are $c$-conformally cobordant if $g \in
% \riem_c(M)$, $h \in \riem_c(N)$ and if there exists a $(n+1)$-dimensional
% compact manifold $\Om$ with boundary $M \amalg N$
% such that $\sigma(\Om;M \amalg N, [g] \amalg [h]) >0$. By Theorem
% \cite{ab} of Akutagawa-Botvinnik, ``to be $c$-conformally cobordant'' is an
% equivalence relation. Obviously, if some metrics $g,h$ on $M$ are
% $c$-concordant then they are also $c$-conformally co 
% \begin{cor} \label{bij2}
% The bijection $\Theta^c_{M,N}$ induce a natural bijection $\tilde{\Theta}_{M,N}: \conf(M) \to \conf(M)$ which maps $\confplus(M)$ onto $\confplus(N)$. In addition, if $M,N,P$ are lds-relative, then $\tilde{\Theta}_{M,N} \circ \tilde{\Theta}_{N,P} = \tilde{\Theta}_{M,P}$.
% \end{cor} 
% \end{proof}

\noindent For $c=0$, Proposition \ref{bij} was already known (see \cite{gajer:93}). The proof here is slightly different and uses only basic facts on surgery. \\

\noindent We now define 
\[ \sigma' \definedas \left| \begin{array}{ccc}
\conc_0(M) & \to & ] \infty,\sigma(M)] \\
C & \mapsto & \sup_{g \in C} \mu(M,[g]).
\end{array} \right. \]
Clearly, 
$$\sup_{C \in \conc_0(M)} \sigma'(C) = \sigma(M).$$
Let also $\sigma'' \definedas \min(\sigma', \beta_n)$. 
As an application of Proposition \ref{bij}, we get Theorem
\ref{application} we recall here:

\begin{cor} \label{cor1}
Assume that $M,N$ are lds-relative, then 
$$\sigma'' (\conc_0(M) )= \sigma''(\conc_0(N)).$$
\end{cor}

%%%%%%%%%%%%%%%%%%%%%%%%%%%%%%%%%%%%%%%%%%%%%%%%%%%%%%%%%
\subsection{Proof of Proposition \ref{bij}}. 
We set 
$$c_n = \min_{k \in \{0,\cdots, n-3\}} \beta_{n,k}>0$$ 
where $\beta_{n,k}$ is the constant which appears in the statement of
Theorem \ref{main}. We fix some $c < c_n$. 
Let $M$, $N$ be some compact manifolds and let $g \in
\riem_c(M)$. Assume that $N$ is
obtained from $M$ by a surgery of dimension $k \in \{0,\cdots,n-3\}$. By
Theorem \ref{adh}, there exists a sequence of metrics
$(g_\th)_{\th >0}$ on $N$ such that for $\th$ small enough (smaller than some $\ep>0$), $g_\th \in
\riem_c(N)$.  
We define  
\[ \Theta^c_{M,N} \definedas \left| \begin{array}{ccc} 
\conc_c(M) &  \to & \conc_c(N) \\

[g]_M^c & \mapsto & [g_\th]_N^c.
\end{array} \right. \]
We have to show that $\Theta^c(M,N)$ is well-defined and is a bijection if
$M$ and $N$ are lds-relative.  
First, let us show that if $0<\th_1, \th_2$ are small enough then 
\begin{eqnarray} \label{s1}
[g_{\th_1}]_N^c = [g_{\th_2}]_N^c.
\end{eqnarray}

\noindent Let $\Om \definedas M \times [0,1]$. We equip $\Om$ with the
product metric $G = g + dt^2$. We attach the $(k+1)$-dimensional handle to
$\Om$ along $M \times \{ 0 \}$ related to the given surgery to obtain a manifold $\Om_1$ with $\partial
\Om_1 = N \amalg M$.  By Theorem \ref{main} applied with $g= G$, there exists a
sequence of metrics $(G^1_\th)$ on $\Om_1$ for which the boundary is
minimal and such that for $\theta$ small,
$\mu(\Om_1, [G^1_\th]; N \amalg M, \partial[G^1_\th]) >0$ and such that
$\partial G^1_\th \in \riem_c(N \amalg M)$. By
construction,
$$\partial G^1_\th = g_\th \amalg g.$$ 
We choose  $\th= \th_1$ small enough so that these conditions are satisfied. 
Now,  we attach the $(k+1)$-dimensional handle to
$\Om$ along $M \times \{ 1 \}$ related to the given surgery to obtain a manifold $\Om_2$ with $\partial
\Om_2 = N \amalg N$. Again by Theorem \ref{main}  applied with
$g=G_{\th_1}^1$, we obtain a sequence of metrics $(G^2_\th)$ on $\Om_2$  for which the boundary is
minimal and such that for $\theta$ small,
$\mu(\Om_2, [G^2_\th]; N \amalg N, \partial[G^2_\th]) >0$, such that
$\partial G^2_\th \in \riem_c(N \amalg N)$  and by
construction,
$$\partial G^2_\th = g_{\th_1} \amalg g_\th.$$ 
Choose $\th_2$ small enough such these conditions are satisfied. Note that
since the metrics $G^1_\th$ is equal to $G$ near $M \times \{1\}$, the
number $\th_2$ does not depend on the choice of $\th_1$. 
Now, by Lemma \ref{surgery_cyl_1}, $N \times [0,1]$ is obtained from
$\Om_2$ by a $(k+1)$-dimensional interior surgery on $\Om_2$. By Theorem
\ref{main_surgery_interior}, there exists a sequence of metrics $(G_\th)$
on $N \times [0,1]$ equal to $G^2_{\th_2}$ in a neighborhood of $N \amalg
N$ such that $\mu(N \times [0,1], [G_\th]; N \amalg N,
\partial[G_\th]) >0$. Since $\partial G^2_{\th_2} = g_{\th_1} \amalg
g_{\th_2}$, we obtain that $\si(N\times[0,1];    N \amalg N,  g_{\th_1} \amalg
g_{\th_2})>0$. Since $\partial G^2_\th \in \riem_c(N \amalg N)$, we have
that $g_{\th_1}, g_{\th_2} \in \riem_c(N)$ and hence, these two metrics are
$c$-concordant. \\

\noindent Now, let $g, h$ be two metrics on $M$ which are $c$-concordant
and let $G$ be a metric on $M \times [0,1]$ such that the boundary $M
\amalg M$ is minimal, with $\partial G = g \amalg h$ and such that 
$\mu(M \times [0,1]; M \amalg M, [g] \amalg [h]) >0$. Doing the same than
above, we show that $g_{\th_1}$ and $h_{\th_2}$ are $c$-concordant on $N$
if $\th_1$ and $\th_2$ are small enough.  
%\mnote{dans les def, verifier que les bords sont minimaux}

\noindent This shows that $\Theta^c_{M,N}$ is well-defined.
Now assume that $M$ and $N$ are lds-relative and consider the dual surgery
from $N$ to $M$. In the same way, we can construct 
$$ \Theta^c_{N,M} : \conc_c(N) \to \conc_c(M)$$
as above. 
We now prove that 
\begin{eqnarray} \label{composition}
\Theta_{N,M}^c \circ \Theta_{M,N}^c = Id_{\conc_c(M)}.
\end{eqnarray}

\noindent Let $g \in \riem_c(M)$. Define $\Om \definedas M \times
[0,1]$ and let $G \definedas g + dt^2$ and let $\Om_1$ be obtained as
above equipped with a metric $G^1_{\th_0}$ ($\th_0$ small enough) for which the boundary is
minimal and such that 
$\partial G^1_{\th_0} = g_{\th_0} \amalg g \in \riem_c(N \amalg M$ with $[g_{\th_0}]^c_N =
\Theta^c_{M,N}([g]_M^c)$ and such that $\mu(\Om_1, [G^1_{\th_0}]; N \amalg
M, \partial[G^1_{\th_0}]) >0$. Now, we attach the $(n-k)$-dimensional handle
on $\Om_1$ along $N$ corresponding to the dual surgery from $N$ to $M$. We
get a new manifold $\Om_3$ such that $\partial \Om_3= M \times M$. We
apply Theorem \ref{main} with $g=G^1_{\th_0}$ and we get a metric $G^3_{\th_3}$
 for which the boundary is
minimal and such that 
$$\partial G^3_{\th_3} = (g_{\th_0})_{\th_3} \amalg g \in \riem_c(M \amalg M)$$ 
with 
\begin{eqnarray} \label{s2} 
[(g_{\th_0})_{\th_3}]_M^c = \Theta^c_{N,M} ([g_{\th_0}]_N^c) =
\Theta_{N,M}^c \circ \Theta_{M,N}^c ([g]_M^c).
\end{eqnarray} 
By Lemma \ref{surgery_cyl_2}, $M \times [0,1]$ is obtained from $\Om_3$ by
an interior $(n-k)$-dimensional surgery. Hence, by Theorem
\ref{main_surgery_interior}, there exists a metric $G_\th$ on $M \times
[0,1]$ equal to $G^3_{\th_3}$ in a neighborhood of the boundary in a neighborhood of $M \amalg
M$ such that $\mu(M \times [0,1], [G_\th]; M \amalg M,
\partial[G_\th]) >0$. Since $\partial G^3_{\th_3} = (g_{\th_0})_{\th_3}  \amalg
g$, and since $(g_{\th_0})_{\th_3} , g  \in \riem_c(N)$, they are
$c$-concordant. By \eref{s2}, we obtain 
$$ [g]_M^c = [(g_{\th_0})_{\th_3}]_M^c = 
\Theta_{N,M}^c \circ \Theta_{M,N}^c ([g]_M^c).$$ \\
This proves Relation \eref{composition}. In the same way, we prove that 
$$\Theta_{M,N}^c \circ \Theta_{N,M}^c = Id_{\conc_c(N)}.$$
We obtain that $\Theta_{M,N}^c$ is a bijective map whose 
inverse is $\Th_{N,M}^c$.\\

\noindent To prove Relation \eref{s3}, we fix $c \leq  c'$ and $h \in \riem_{c'}(M)$.
In view of  the definition of $\Th^c_{M,N}$ and using Relation \eref{cc'}, 
we have for $\th$ small enough 
$$\Theta^{c'}_{M,N}([h]^{c'}_M) = 
[h_\th]_N^{c'} = [h_\th]_N^c \cap \riem_{c'} (N) =
\Theta^c_{M,N}([h_\th]_M^c) \cap \riem_{c'} (N).$$
The proof of Proposition \ref{bij} is now complete. 

\subsection{Proof of Corollary \ref{cor1}}
Let $C \in \conc_0(M)$, $C' \definedas \Th^0_{M,N}(C)$. 
Set $c \definedas \si''(C)$ and $c' \definedas \si''(C')$. 
We are done if we prove that
\begin{eqnarray} \label{CC'}
c=c'.
\end{eqnarray}
By definition of $\sigma''$, for all $\ep >0$,  $C
\cap \riem_{c -\ep} \not= \emptyset$. So let $h_\ep \in C
\cap \riem_{c -\ep}$. 
By Relation \eref{cc'}, $C \cap \riem_{c-\ep} = [h_\ep]_M^{c-\ep}$. 
Relation \eref{s3} then leads to 
\begin{eqnarray*}
C' \cap \riem_{c-\ep}(N) & = & \Theta^0_{M,N} (C)
\cap \riem_{c- \ep}(N) \\
 & = & \Theta^0_{M,N} ( [h_\ep]_M^0)
\cap \riem_{c- \ep}(N) \\
& =& \Theta^{c-\ep}_{M,N} ([h_\ep]_M^{c-\ep})\\
& =&  \Theta^{c-\ep}_{M,N} (C \cap \riem_{c-\ep}(M))
\end{eqnarray*}
and consequently,   
$C' \cap \riem_{c-\ep}(N) \not= \emptyset$ which implies $c' \geq c$. 
In the same way, since $\Th^0_{M,N} = (\Th^0_{N,M})^{-1}$, we have $c \geq c'$ and Relation \eref{CC'} is proven. This
ends the proof of Corollary \ref{cor1}.

%%%%%%%%%%%%%%%%%%%%%%%%%%%%%%%%%%%%%%%%%%%%%%%%%%%%%%%%%%%%%%%%%%%%%

%%%%%%%%%%%%%%%%%%%%%%%%%%%%%%%%%%%%%%%%%%%%%%%%%%%%%%%

%%%%%%%%%%%%%%%%%%%%%%%%%%%%%%%%%%%%%%%%%%%%%%%%%%%%%%%%%%%%%%%%%%%%%%%%

%%%%%%%%%%%%%%%%%%%%%%%%%%%%%%%%%%%%%%%%%%%%%%%%%%%%%%%%%%%%%%%%%
%%%%%%%%%%%%%%%%%%%%%%%%%%%%%%%%%%%%%%%%%%%%%%%%%%%%%%%%%%%%%%%%%
%\bibliographystyle{amsplain}
%\bibliography{literatur}
\end{document}